\documentclass[a4paper,12pt,reqno]{amsart}
\pdfoutput=1

\usepackage{mathrsfs,mathdots,amssymb}
\usepackage[T1]{fontenc}
\usepackage[utf8]{inputenc}
\usepackage[centering]{geometry}

\usepackage{hyperref}

\usepackage[numbers,sort&compress]{natbib}

\theoremstyle{plain}

\newtheorem{theorem}{Theorem}[section]
\newtheorem*{theorem*}{Theorem}
\newtheorem{corollary}[theorem]{Corollary}
\newtheorem{lemma}[theorem]{Lemma}
\newtheorem{proposition}[theorem]{Proposition}
\newtheorem{example}[theorem]{Example}
\numberwithin{equation}{section}

\theoremstyle{remark}

\newtheorem{definition}[theorem]{Definition}
\newtheorem{remark}[theorem]{Remark}

 \def\C{\mathbb C}
   
\def \N{\mathbb N} 
\def\i{{\rm i}}

\def\wtilde{\widetilde}

\newcommand{\coloneqq}{\mathrel{\mathop:}=}

\newcommand{\R}{\mathbb{R}}

\newcommand{\rank}{\operatorname{rank}}

\title[Herglotz-Nevanlinna matrix functions]%
{Herglotz-Nevanlinna matrix functions and  Hurwitz stability of matrix polynomials}
\date{\today}


\author[X. Zhan]{Xuzhou Zhan}
\address[X. Zhan]{College of education for the future, Beijing Normal University,
   Zhuhai 519087, China}
\email{91122021027@bnu.edu.cn}

\begin{document}

\begin{abstract}
This paper elaborates on a relationship between  matrix-valued Herglotz-Nevanlinna functions and Hurwitz stable matrix polynomials, which generalizes the corresponding classical stability criterion.  
The main motivation comes from the author's recent stability studies linked with matricial Markov parameters.  
To  fulfill our goals,  we first give a partial-fraction decomposition of a self-adjoint rational matrix function with the Herglotz-Nevanlinna property.
The next step is to connect a matrix-valued Herglotz-Nevanlinna function with its matricial Laurent series. Certain matrix extensions to two classical theorems by Chebotarev  and Grommer, respectively, are also established.\\[.3em]


\end{abstract}

\maketitle

\section{Introduction}\label{SectionIntro}
 
 Throughout this paper, we denote by~$\C$,  $\R$, $\N$ and~$\N_0$, respectively, the sets
of all complex, real, positive integer and nonnegative integer numbers.
Unless explicitly
noted, we assume in this paper that $p,q,m,n\in \N$.
Let~$\mathbb C^{p\times q}$ stand for the set
of all complex~$p\times q$ matrices. Let also~$0_p$ and~$I_p$ be, respectively, the zero and the
identity~$p\times p$ matrices. Given a matrix~$A\in \mathbb C^{p\times p}$ we denote its transpose by~$A^{\rm T}$ and
its conjugate transpose by~$A^*$. If $A$ is self-adjoint ($A=A^*$), we write
$$
A\begin{cases} \succ 0, &  \mbox{if }
A \mbox{ is positive definite};\\
\prec 0, & \mbox{if } A \mbox{ is negative definite};\\
\succeq 0, & \mbox{if } A \mbox{ is nonnegative definite};\\
\preceq 0, & \mbox{if } A \mbox{ is nonpositive definite}.
\end{cases}
$$
 
Denote  by $\mathbb C_+$ the open upper half of the complex plane. Recall that a  function $R:\mathbb C_+\rightarrow \mathbb C^{p\times p}$ is said to be a {\it matrix-valued Herglotz-Nevanlinna} \footnote{In the scalar case $p=1$, other popular titles for the same class of functions are ``Nevanlinna'',  ``Pick'', ``Nevanlinna-Pick'', ``Herglotz'' and ``R-functions''.}  {\it  function} or a matrix-valued  function with the {\it Herglotz-Nevanlinna property} if it is holomorphic on $\mathbb C_+$ and its  imaginary part satisfies that
$$
{\rm Im}R(z)=\frac{1}{2{\rm i}}\left(R(z)-R(z)^{*}\right)\succeq 0,\quad z\in \mathbb C_+.
$$
The applications of this matrix analogue of scalar Herglotz-Nevanlinna functions (analytic maps of $\mathbb C_+$ into itself)
  or their operator-valued analogue appear in  system  theory \cite{GT00,BS}, quantum walks \cite{GrVe}, interpolation and moment problems \cite{Kov}, etc.  Some other extensions  of  Herglotz-Nevanlinna functions or their variants  can also be seen in the formulations of noncommutative functions \cite{PT} and multivariate functions \cite{AACLS,ACS,LN}. 
  
Given a $p\times p$ matrix polynomial $F(z)$ of the form
\begin{equation}\label{ReprentationMatrixPolynomials}
    F(z)=\sum_{k=0}^{n} A_k z^{n-k},\quad \mbox{with}\quad A_0,\ldots, A_n\in \mathbb C^{p\times p},
    \quad A_0\ne 0_{p},
\end{equation}
$n$ is called the \emph{degree} of~$F(z)$ and denoted by $\deg F$.  $F(z)$ is said to be \emph{regular} if
$\det F(z)$ is not identically zero and it is
\emph{monic} if~$A_0=I_p$. 
Given
a regular matrix polynomial $F(z)$, we say that $\lambda\in \mathbb C$ is a \emph{zero} (also called \emph{latent root}) of
$F(z)$ if the determinant $\det F(\lambda)=0$. The \emph{spectrum}~$\sigma(F)$ of~$F(z)$ is the set of all zeros of~$F(z)$. 

 In particular, a regular $p\times p$ matrix polynomial $F(z)$ is called {\it Hurwitz stable} if $\sigma(F)$  is a subset of the open left half of the complex plane. For application,  
 the Hurwitz stability of $F(z)$ determines the asymptotic stability of the linear high-order  differential systems
\begin{equation*}
    A_{0} y^{(n)}(t)+A_{1} y^{(n-1)}(t)+\cdots+A_{n} y(t)=u(t),
\end{equation*}
where $y(t)$ and $u(t)$ are, respectively, the control output vector and the control input vector with respect to the time parameter $t$.  For testing the Hurwitz stability for regular matrix polynomials without
computing their determinants or zeros, one can use algebraic techniques as in the literature  \cite{HAP,LPJ, LT82,BA,Gal,HH,XD}.
  
For a real scalar polynomial, its Hurwitz stability -- the property that all its zeros lie in the open left half-plane -- has an intimate connection with the Herglotz-Nevanlinna property of a related real rational function:

 \begin{theorem*}[\textit{Stability criterion via Herglotz-Nevanlinna functions}] \cite[Proposition 4.4]{AGT} \cite{Gan,Bar} 
 Let $f(z)$ be a real scalar polynomial  of degree~$n$ with the even part $f_{e}(z)$ and odd part $f_{o}(z)$ which satisfies $f_e(z^2)+z f_o(z^2)=f(z)$. $f(z)$  is  Hurwitz stable if and only if the associated rational function $r(z)=-\frac{f_{o}(z)}{f_e(z)}$ is a Herglotz-Nevanlinna function with 
exactly $\lfloor \frac{n}{2}\rfloor$ poles, all of which are negative real, and the limit $
\lim\limits_{z\rightarrow \infty} r(z)$ is positive real when $n$  is additionally odd.
\end{theorem*}

To extend the stability criterion via Herglotz-Nevanlinna functions to the matrix case, we are going to clarify the relation between the Hurwitz  stability for matrix polynomials and  matrix-valued Herglotz-Nevanlinna functions. More specifically,  given  a matrix polynomial $F(z)$, one needs to find an associated {\it  rational matrix function} (that is,  a matrix-valued function whose
entries are complex-valued rational functions) and connect the Hurwitz stability of $F(z)$ with its Herglotz-Nevanlinna property.

For a rational matrix function $R(z)$,  
the author is unaware of certain suitable decompositions or forms to reflect all its zeros and poles. On the other hand, turning $R(z)$ into some scalar rational function (the determinant $\det R(z)$, the trace of $R(z)$, $X^* R(z) Y$ for some column vectors $X$ and $Y$ etc.) in general results in loss of some of its essential features. 
 That is to say, certain basic techniques in the scalar proofs may be totally unavailable to investigate Herglotz-Nevanlinna property and  Hurwitz stability in the matrix case. 

To overcome these obstacles, we pose a new line to deal with the matrix extension.
Our idea is to investigate under what conditions a rational matrix function is a matrix-valued Herglotz-Nevanlinna function.  We begin with a partial-fraction decomposition of $R(z)$ when it becomes a matrix-valued Herglotz-Nevanlinna function (see Proposition \ref{ProRRepresent}). Here we focus on the case when $R(z)$ is {\it self-adjoint}, that is, it obeys that $R(z)=R(\bar z)^*$ 
for all $z\in\mathbb C$ except the poles of  the entries of $R(z)$. This factorization with  constant coefficient matrices given in Corollary \ref{CorMassE_j} can be viewed as a matricial analogue of the classical Chebotarev theorem.

The next step is to link the Herglotz-Nevanlinna property of $R(z)$ with the matricial Laurent series 
\begin{equation}\label{RationalLaurent}
       R(z)=\sum_{j=0}^{k}z^{j} {\bf s}_{-(j+1)}+\sum_{j=0}^{\infty} \frac{{\bf s}_j}{z^{j+1}} 
     \end{equation}
    that converges for sufficiently large $|z|$.  More specifically,
denote a finite or infinite block Hankel matrix associated with~$R(z)$ by
\begin{equation*}
    {\mathscr H}_{j,k}(R)\coloneqq\begin{bmatrix}
        {\bf s}_j   & \cdots & {\bf s}_{j+k}  \\
        \vdots   & \ddots & \vdots\\
        {\bf s}_{j+k} &  \cdots& {\bf s}_{j+2k}
    \end{bmatrix},
\end{equation*}
where~$j \in\N_0\cup \{-1,-2\}$ and~$k\in \N_0\cup\{\infty\}$. For simplicity we
write~${\mathscr H}_{k}(R)$ for~${\mathscr H}_{0,k}(R)$. The above  block Hankel matrices  satisfy some properties when $R(z)$ is a matrix-valued Herglotz-Nevanlinna function (the definition of {\it right coprimeness} of matrix polynomials can be found in Definition \ref{DefCoprime} below): 

\begin{theorem}\label{ProRHankel}
Let $R(z)$ be a $p\times p$ self-adjoint rational matrix  function with the matrix fraction \eqref{RQP}. Assume that $R(z)$ is  a matrix-valued Herglotz-Nevanlinna function. Then ${\mathscr H}_{-2,0}(R)\succeq 0$, ${\mathscr H}_{{\rm deg}P-1}(R)\preceq 0$  and $|\deg Q-\deg P|\leq 1$. If, in addition, $P(z)$ and $Q(z)$  are right coprime and $Q(z)$ is regular, then ${\mathscr H}_{{\rm deg}P-1}(R)\prec 0$. In this case, $A\equiv 0_p$ or $A \succ 0$ for $R(z)$ written
in the form \eqref{RRepresent}. 
\end{theorem}

Conversely, for a self-adjoint rational matrix  function $R(z)$,  these block Hankel matrices  act as a tool to check the Herglotz-Nevanlinna property of $R(z)$ (see the following theorem).  In this event,  the denominator matrix polynomial  $P(z)$ in the matrix fraction \eqref{RQP}  is shown to be the so-called \emph{simple matrix polynomial} (see Definition \ref{DefSimple} and Proposition \ref{ProSimple}), which indicates that the roots of its invariant polynomials are limited to be simple although the zeros of $P(z)$ still may be multiple. 
 
\begin{theorem}\label{ProHankelR}
 Let $R(z)$ be a $p\times p$ self-adjoint rational matrix  function with the matrix fraction \eqref{RQP}. Suppose that
${\mathscr H}_{-2,0}(R)\succeq 0$,  ${\mathscr H}_{{\rm deg}P-1}(R)\prec 0$ and $\deg Q-\deg P\leq 1$. Then 
\begin{enumerate}
\item[(i)]  $R(z)$ is  a matrix-valued Herglotz-Nevanlinna function;
\item[(ii)]  $|\deg Q-\deg P|\leq 1$;
\item[(iii)] $P(z)$ is a simple matrix polynomial and $\sigma(P)\subseteq \mathbb R$;
\item[(iv)]  $Q(z)$ and $P(z)$ are right coprime.
\end{enumerate}
\end{theorem}

When $R(z)$ is a real scalar rational function, a combination of Theorems \ref{ProRHankel} and \ref{ProHankelR} reduces to a classical result obtained by Grommer \cite{Gro} (see Section \ref{SectionGrommer}). For more details on this so-called \emph{Grommer theorem} (\emph{for rational functions}) we also refer the reader to \cite{AK}, \cite{CM}, \cite[Theorem 3.4]{Hol}, etc.

The extension of Grommer's theorem  helps us to get closer to our main goal, which is greatly 
motivated by the recent study \cite{XD}. 
Given a monic matrix polynomial $F(z)$,  \cite[Theorems 4.4 and 4.10]{XD} finds some associated rational matrix functions $R(z)$ which admit the Laurent series as in \eqref{RationalLaurent}.  The Hurwitz stability of $F(z)$ can be tested through the coefficients ${\bf s}_k$, which are referred to as {\it matricial Markov parameters}  (see Definition 2.2 of \cite{XD}) of $F(z)$.  Another stability criterion via matricial Markov parameters addresses the situation when $\deg F$ is even (see Theorem \ref{ThmAlterEven}).  By means of the extended Grommer theorem, these stability criteria for matrix polynomials can be reshaped in terms of the Herglotz-Nevanlinna property of $R(z)$ (see Theorems \ref{ThmHurwitz} and \ref{ThmHurwitzodd}).

We conclude the introduction with the outline of the paper. In Section~\ref{SectionGrommer} we obtain a matricial analogue of the Chebotarev theorem and then prove Theorems \ref{ProRHankel} and \ref{ProHankelR}.
 The main results are obtained in
Section~\ref{SectionStability}: Based on Theorems \ref{ProRHankel} and \ref{ProHankelR}, we convert the stability criteria for matrix polynomials  via matricial Markov parameters  to that  via matrix-valued Herglotz-Nevanlinna functions.

 \section{Proof of  Grommer theorem for  rational matrix functions}
\label{SectionGrommer}

Grommer \cite{Gro} characterizes a real rational function $r(z)$ to be a Herglotz-Nevanlinna function in terms of its Laurent series  (he also considers more general case that $r(z)$ is a real meromorphic function):
 \begin{theorem*}[\textit{Grommer theorem for rational functions}]  \cite{AK,CM,Gro} \cite[Theorem 3.4]{Hol}
Let  $p(z)$ and $q(z)$ be two coprime real polynomials satisfying that $|\deg p-\deg q|\leq 1$.  Suppose that  $r(z):=\frac{q(z)}{p(z)}$ is a real rational function.  $r(z)$ is  a Herglotz-Nevanlinna function if and only if  $r(z)$ can be represented by the Laurent series 
 $$
   r(z)=s_{-2} z+s_{-1}+\sum_{k=0}^{\infty} \frac{ s_{k}}{z^{k+1}} 
    $$  
    with the conditions that $s_{-2}\geq 0$ and the Hankel matrix $[s_{j+k}]_{j,k=0}^{\deg p-1}$ is negative definite.
\end{theorem*}

When $R(z)$ is  a real scalar rational function, Theorem \ref{ProRHankel} reduces to  the ``only if'' implication of the Grommer theorem for rational functions. On the other hand, Theorem \ref{ProHankelR} extends the ``if'' implication to the corresponding results for rational matrix functions.

This section is devoted to prove Theorems \ref{ProRHankel} and \ref{ProHankelR}. We begin with a partial-fraction decomposition of rational matrix functions with the Herglotz-Nevanlinna property.  

\begin{definition}\label{DefCoprime}
Given two $p\times p$ matrix polynomials $F(z)$ and $G(z)$, if there exists a pair of $p\times p$ matrix polynomials $C(z)$ and $E(z)$ such that
$$F(z)=C(z)G(z)+E(z)$$
and 
$\deg E <\deg G$, then we call  $C(z)$ and $E(z)$ are the {\it right quotient} and {\it right remainder}  of $F(z)$ on {\it division} by $G(z)$.
If $C(z)$ and $0_p$ are the right quotient and right remainder  of $F(z)$ on division by $G(z)$, then  $G(z)$ is called a \emph{right}
    \emph{divisor} of $F(z)$.
        Let additionally $\wtilde F(z)$ be a $p\times p$ matrix polynomial. Then
    \begin{itemize}
\item   ~$G(z)$ is called a
    \emph{right}  \emph{common divisor} of~$F(z)$ and $\wtilde F(z)$ if
    $L(z)$ is a right  divisor of $F(z)$ and also a right  divisor of
    $\wtilde F(z)$.
    \item $G(z)$ is called a \emph{greatest right} 
    \emph{common divisor} (\emph{GRCD}) \emph{of $F(z)$ and} $\wtilde F(z)$
    if any other right  common divisor is a right  divisor of $L(z)$.
    \item $F(z)$ and $\wtilde F(z)$ are said to be \emph{right coprime} if any  GRCD of $F(z)$ and $\wtilde F(z)$ is \emph{unimodular}, that is, its determinant never vanishes in~$\mathbb C$.
\end{itemize}   
\end{definition}

Let $R(z)$ be a $p\times p$ rational matrix function.
Given $A$, $B\in \mathbb C^{p\times p}$ such that $B$ is nonsingular, denote 
$
\frac{A}{B}\coloneqq A\cdot B^{-1}.
$
 According to \cite[Section 6.1]{Kai}, one can write $R(z)$  as a fraction of a $p\times p$ matrix polynomial $Q(z)$ and a monic $p\times p$ matrix polynomial $P(z)$ (Here we may, without loss of generality, assume that $\deg P\geq 1$.)
\begin{equation}\label{RQP}
R(z):=\frac{Q(z)}{P(z)},\quad z\not\in \sigma(P).
\end{equation}
 In view of  \cite[P. 78]{Gan}, there exists a unique pair of right quotient $R_p(z)$ and right remainder  $\wtilde Q(z)$ of $Q(z)$ on division by $P(z)$. Lemma 6.3-11 of \cite{Kai} tells that the  rational matrix function $R_{sp}(z):=\wtilde Q(z)(P(z))^{-1}$
is  {\it strictly proper} , that is, $\lim \limits_{z\rightarrow \infty} R_{sp}(z)=0_p$. 
Thus $R(z)$ can be uniquely decomposed into its polynomial part and strictly proper part:
\begin{equation}\label{sumR}
R(z)=R_p(z)+R_{sp}(z).
\end{equation}
Moreover, if $R(z)$ is self-adjoint, the entry relation in this decomposition indicates that both $R_p(z)$ and $R_{sp}(z)$ are self-adjoint as well.
By contrast to \cite[Theorem 4.1, pp. 666-667]{BA},  this decomposition does not assure that $\wtilde Q(z)$ is  regular.

Let $R(z)$ be a $p\times p$ Herglotz-Nevanlinna matrix function. It is well known (e.g. \cite{Kov}) that $R(z)$ admits the so-called Nevanlinna-Riesz-Herglotz integral  representation
\begin{equation}\label{RepresentHN}
R(z)= \wtilde A z+\wtilde B+\int_{\mathbb R}(\frac1{u-z}-\frac{u}{1+u}){\rm d}\Omega(u),\quad z\in \mathbb C_+,
\end{equation}
where $ \wtilde A \succeq 0,\ \wtilde B=\wtilde B^*$ and $\Omega$ is a ${p\times p}$ positive semi-definite Borel matrix measure on $\mathbb R$ such that
$$
{\rm trace}\int_{\mathbb R} \frac{1}{1+u^2}{\rm d}\Omega(u)<+\infty,
$$
which is called the spectral measure of $R(z)$. 
\cite[Lemma 5.6]{GT00} discusses how  $R(z)$ analytically continues through an interval  from $\mathbb C_+$ into $\mathbb C_-$. Special attention is put to the analytic continuation of $R(z)$ through an interval $(\lambda_1,\lambda_2)$ on $\mathbb R$ by reflection, viz.,
$
R(z)=R(\bar z)^*$  for all  $z\in \mathbb C_-.
$
For the case when $R(z)$ is also a (necessarily self-adjoint) rational matrix function with the matrix fraction \eqref{RQP}, this continuation can be estabilished once  the choosen  interval $(\lambda_1,\lambda_2)$ excludes all zeros of $P(z)$.   

The following proposition looks for which representation is needed for a  self-adjoint rational matrix  function is  to be a matrix-valued Herglotz-Nevanlinna function with analytic continuation.

\begin{proposition}\label{ProRRepresent}
Suppose that  $R(z)$ is a $p\times p$ self-adjoint rational matrix  function with the matrix fraction \eqref{RQP}. 
$R(z)$ is a matrix-valued Herglotz-Nevanlinna function  if and only if
$R(z)$ can be represented in the form 
\begin{equation}\label{RRepresent}
R(z)=A z+ B+\sum_{j=1}^r \frac{E_j}{\lambda_j-z},\quad z\not\in \{\lambda_j\}_{j=1}^r,
\end{equation}
where $A \succeq 0,\ B=B^*$, $E_j\succeq 0$  and $\{\lambda_j\}_{j=1}^r \subseteq \sigma(P)\cap \mathbb R$. In this case, if $\sigma(P)\subseteq \mathbb R$, then $\{\lambda_j\}_{j=1}^r$ coincides with $\sigma(P)$ (some $E_j$ related to $\lambda_j$ may be $0_p$). 
\end{proposition}

\begin{proof}
Suppose that $R(z)$ is a matrix-valued Herglotz-Nevanlinna function of the form \eqref{RepresentHN}. Observe that the right-hand side of the formula \eqref{RepresentHN} is also well-defined in $\mathbb C_-$ and, moreover, extends to a wider region $\mathbb C\setminus \mathbb R$ as a self-adjoint matrix function. On the other hand,  the rational matrix function $R(z)$  is self-adjoint and analytic in $\mathbb C\setminus \sigma(P)$, so \eqref{RepresentHN} holds in $\mathbb C\setminus  \sigma(P)$ as well. That is to say, the support of the matrix measure $\Omega$, which is denoted by 
 $\{\lambda_j\}_{j=1}^r$ below, is a subset of $\sigma(P)$.
Then we can rewrite  \eqref{RepresentHN} into
\begin{equation*}
R(z)= \wtilde A z+(\wtilde B-\sum_{j=1}^r \Omega(\{\lambda_j\})\frac{\lambda_j}{1+\lambda_j} )+\sum_{j=1}^r  \frac{\Omega(\{\lambda_j\})}{\lambda_j-z},\quad z\not\in\{\lambda_j\}_{j=1}^r,
\end{equation*}
and, subsequently, into \eqref{RRepresent}.

 If $R(z)$ has the form  \eqref{RRepresent}, then $R(z)$ is analytic  in  $\mathbb C_+$ and
$$
\dfrac{R(z)-R(z)^*}{z-\bar z}=C+\sum_{j=1}^r\dfrac{E_j}{|\lambda_j-z|^2}\succeq 0,\quad z\in \mathbb C_+.
$$
Thus $R(z)$ is a matrix-valued Herglotz-Nevanlinna function.

In this case,  suppose that $\sigma(P)\subseteq \mathbb R$ and $\lambda\in \sigma(P)$.
If $\Omega(\{\lambda\})\neq 0_p$, then $\lambda\in \{\lambda_j\}_{j=1}^r$. Conversely, denote $\lambda$ by $\lambda_{r+1}$ and set its corresponding matrix $E_{r+1}$ as in \eqref{RRepresent} to be $0_p$.
\end{proof}

This proposition can be viewed as a matricial analogue to the classical Chebotarev theorem (cf.  \cite{CM,Tsc},  \cite[Theorem 3.4]{Hol}), which however in the scalar case does not recover the full information of the latter: Each matrix mass $E_j$ is still unknown. Under certain additional conditions, its  representation will be given in Corollary \ref{CorMassE_j}.

\begin{definition}
Given a quadruple of  $p\times p$ matrix polynomials~$L(z),\wtilde L(z),M(z)$ and $\wtilde M(z)$
satisfying
\begin{equation*}
    \wtilde M(z)\wtilde L(z)=M(z)L(z),
\end{equation*}
the \emph{associated Anderson-Jury Bezoutian matrix} ${\mathbf B}_{\wtilde M,M}(L,\wtilde L)$
 is defined via the formula
\begin{equation*}
\begin{bmatrix}I_p & z I_p & \cdots & z^{n_1-1}I_p \end{bmatrix}\cdot
    {\mathbf B}_{\wtilde M,M}(L,\wtilde L)\cdot
    \begin{bmatrix}I_p \\ u I_p \\ \vdots \\ u^{n_2-1}I_p \end{bmatrix}%^{\rm T}
    =\frac1{z-u} \left(\wtilde M(z)\wtilde L(u)-M(z)L(u)\right), 
\end{equation*}
where $n_1\coloneqq\max\{\deg M, \deg \wtilde M\}$ and
$n_2\coloneqq\max\{\deg L, \deg \wtilde L\}$.
\end{definition}

Given an Anderson-Jury Bezoutian matrix~${\mathbf B}_{\wtilde M,M}(L,\wtilde L)$, it is natural to choose~$\wtilde M(z)=L(z)$ and~$M(z)=\wtilde L(z)$ when~$L(z)\wtilde L(z)=\wtilde L(z)L(z)$. For a nontrivial
choice of $\wtilde M(z)$ and $M(z)$ in the general non-commutative case, we refer the reader to the
construction of the common multiples via spectral theory of matrix polynomials:
see~\cite[Theorem 9.11]{GLRMP} for the monic case and~\cite[Theorem 2.2]{GKLR} for the comonic
case. For more comprehensive study of Anderson-Jury Bezoutiants, we refer the reader to \cite{AnJu,LT82,GLRMP,GKLR}. 

We are in a position to prove Theorem \ref{ProRHankel}  with the following lemma.  

\begin{lemma}\label{LemBezoutHankel}
Let $R(z)$ be a $p\times p$ self-adjoint rational matrix  function with the matrix fraction \eqref{RQP}. Suppose that $\deg Q-\deg P \leq 1$.
Assume that  $(R_p(z), \wtilde Q(z))$ is the unique pair of self-adjoint right quotient and right remainder  on division by $P(z)$. Then 
$\deg R_p(z)\leq 1$. Suppose that
   $R_p(z)$ and $P(z)$
are written in the form
    \begin{equation*}
    R_p(z)=Az+B
      \quad\text{and}\quad
        P(z)=\sum_{k=0}^{m} P_{m-k} z^k,
    \end{equation*}
   where $A=A^*$, $B=B^*$ and $m:=\deg P$.
   \begin{enumerate}
   \item[(a)] In the case when $A\neq 0_p$,     
   \begin{equation*}
   B_{P^{\vee},Q^{\vee}}(P,Q)= \begin{bmatrix}
P_{m}^* & \cdots & P_{0}^*\\
\vdots & \iddots & \\
P_{0}^* &&
\end{bmatrix}\cdot
\begin{bmatrix}-A & \\ & \mathscr H_{m-1}(R) \end{bmatrix}\cdot
\begin{bmatrix}
P_{m} & \cdots & P_{0}\\
\vdots & \iddots & \\
P_{0} &&
\end{bmatrix}.  
\end{equation*}
Moreover, $A$ and ${\mathscr H}_{m-1}(R)$ are of full rank if and only if $P(z)$ and $Q(z)$ are right coprime and $Q(z)$ is regular.
\item[(b)] In the case when $A=0_p$,  
  \begin{equation*}
   B_{P^{\vee},Q^{\vee}}(P,Q)= \begin{bmatrix}
P_{m-1}^* & \cdots & P_{0}^*\\
\vdots & \iddots & \\
P_{0}^* &&
\end{bmatrix}\cdot
 \mathscr H_{m-1}(R)  \cdot
\begin{bmatrix}
P_{m-1} & \cdots & P_{0}\\
\vdots & \iddots & \\
P_{0} &&
\end{bmatrix}.  
\end{equation*}
Moreover, suppose that $Q(z)$ is regular. ${\mathscr H}_{m-1}(R)$ is of full rank if and only if $P(z)$ and $Q(z)$ are right coprime.
\end{enumerate}
\end{lemma}

\begin{proof} We only give a proof for the statement (a). Obviously when $\deg Q-\deg P \leq 1$,   $\deg R_p(z)\leq 1$. By Lemma 3.4 of \cite{XD},   
\begin{align*}
  \begin{bmatrix}I_p & \cdots & z^{m}I_p\end{bmatrix}&\cdot
  % \begin{bmatrix}I_p\\ \vdots\\z^{m}I_p\end{bmatrix}^{\rm T}
  \begin{bmatrix}
P_{m}^* & \cdots & P_{0}^*\\
\vdots & \iddots & \\
P_{0}^* &&
\end{bmatrix}\cdot
\begin{bmatrix}-A & \\ &  {H}_{m-1}(R) \end{bmatrix}\cdot
\begin{bmatrix}
P_{m} & \cdots & P_{0}\\
\vdots & \iddots & \\
P_{0} &&
\end{bmatrix}\cdot
\begin{bmatrix}I_p\\ \vdots\\ u^{m}I_p\end{bmatrix}
  \\
  &=\begin{bmatrix}I_p & \cdots & z^{m-1}I_p\end{bmatrix}
    {\mathbf B}_{P^{\vee},\wtilde Q^{\vee}}\left(P,\wtilde Q\right)
    \begin{bmatrix}
        I_p\\ \vdots\\ u^{m-1}I_p
    \end{bmatrix}
  -P^{\vee}(z)A  P(u)\\
&=\frac{1}{z-u}\left(P^{\vee}(z)Q(u)-Q^{\vee}(z)P(u)\right)\\
&=\begin{bmatrix}I_p & \cdots & z^{m}I_p\end{bmatrix}
  {\mathbf B}_{P^{\vee},Q^{\vee}}\left(P,Q\right)
  \begin{bmatrix}I_p\\ \vdots\\ u^{m}I_p
  \end{bmatrix}.
\end{align*}
It follows from Theorem 0.2 in  \cite{LT82} that $A$ and ${\mathscr H}_{m-1}(R)$ are of full rank if and only if $P(z)$ and $Q(z)$ are right coprime and $Q(z)$ is regular. 
\end{proof}

\begin{proof}[Proof of Theorem \ref{ProRHankel}]
Taking Proposition \ref{ProRRepresent} into account, one can represent
$R(z)$ 
in the form \eqref{RRepresent}.  That is to say, ${\mathscr H}_{-2,0}(R)=C\succeq 0$ and the form \eqref{RRepresent} and decomposition \eqref{sumR} of $R(z)$ are related by
\begin{equation*}
R_p(z)=Az+B,\quad R_{sp}(z)=\sum_{j=1}^r \frac{E_j}{\lambda_j-z}.
\end{equation*}
By \eqref{sumR} one can see that $Az+B$ is the right quotient of $Q(z)$ on division by $P(z)$. It means that  $\deg Q\leq \deg P+1$.

The proof for the statement $\deg Q\geq \deg P-1$  is conducted by contradiction.
By assuming that $\deg Q-\deg P<-1$,  Lemma 6.3-11 of \cite{Kai} tells that $R(z)$ is strictly proper and so is $zR(z)$. However, these two statements contradict with each other under our assumption: The former statement  implies that  $R(z)=R_{sp}(z)=\sum_{j=1}^r \frac{E_j}{\lambda_j-z}$. It follows that 
 $$
\lim_{z\rightarrow \infty} zR(z)=-\sum_{j=1}^r E_j\neq 0_p,
$$
which contradicts with the latter statement. 

Assume that $m:=\deg P$. The form \eqref{RRepresent} guarantees that $H_{\infty}(R)$ is a  sum of the nonpositive definite matrices $H_{\infty}(\frac{E_j}{z-\lambda_j})$, so it is necessarily a nonpositive definite matrix.
Therefore, ${\mathscr H}_{m-1}(R)\preceq~0$. By Lemma \ref{LemBezoutHankel},  ${\mathscr H}_{m-1}(R)\prec~0$ if $P(z)$ and $Q(z)$ are right coprime. In this case, $A\equiv 0_p$ or $A \succ 0$.

\end{proof}

For the proof of Theorem \ref{ProHankelR}, let us recall some special matrix polynomials.

 \begin{definition}\label{DefSimple}\cite[P. 42]{Lan}
Let $F(z)$ be a $p\times p$ matrix polynomial. $F(z)$ is said to be {\it simple} if $F(z)$ is regular and for any  zero $\lambda$ of $F(z)$, the multiplicity of $\lambda$ coincides with  the degeneracy of $F(z)$ evaluated at $\lambda$, i.e., the nullity of the matrix $F(\lambda)$.
\end{definition}

The reason why a matrix polynomial as in the above definition is called ``simple'' can be seen from its Smith form.
Recall that for a $p\times p$ regular matrix polynomial $F(z)$ of degree $n$, there exist two  $p\times p$ unimodular matrix polynomials $E_L(z)$ and $E_R(z)$ such that $F(z)$ reduces to the following {\it Smith form}:
$$
E_L(z)F(z)E_R(z)={\rm diag}[f_1(z), f_2(z) , \cdots,  f_{p}(z)],
$$
where $f_1(z),\ldots,f_{p}(z)$ are monic scalar polynomials uniquely determined by $F(z)$ and, for $j=1,\ldots,p-1$, $f_{j+1}(z)$ is divisible by $f_j(z)$.
The factors $f_j(z)$ are called {\it invariant polynomials} of $F(z)$. Moreover, if $F(z)$ has $r$ distinct zeros $\lambda_1,\ldots,\lambda_r$, write each invariant polynomial $f_j(z)$ as
$$
f_j(z)=(z-\lambda_1)^{l_{j,1}}(z-\lambda_2)^{l_{j,2}}\cdots(z-\lambda_r)^{l_{j,r}} 
$$
Then the factors $(z-\lambda_k)^{l_{j,k}}$ are called the {\it elementary divisors }of $F(z)$.

\begin{proposition}\label{ProSimple} \cite[Corollary 1, P. 46]{Lan}
A $p\times p$ matrix polynomial is simple if and only if all its elementary divisors are linear in $\mathbb C$, or equivalently, all zeros of its invariant polynomials are simple.
\end{proposition}

Certain properties of a simple matrix polynomial may also be found via derivatives of its adjoint matrix: 
\begin{proposition}\cite[Lemma 2.2]{Du} \label{ProSimpleAdj}
Let $F(z)$ be a $p\times p$ simple matrix polynomial and let $\lambda$ be a zero of $F(z)$ with multiplicity $l$. Then $({\rm adj}\ F)^{(k)}(\lambda)=0_p$ for $k=0,1,\ldots, l-2$ and $({\rm adj}\ F)^{(l-1)}(\lambda)\neq 0_p$. Moreover, $\rank ({\rm adj}\ F)^{(l-1)}(\lambda)=~l$.
\end{proposition}

Given a $p\times p$ matrix polynomial~$F(z)$ written as
in~\eqref{ReprentationMatrixPolynomials}, define a matrix polynomial $F^{\vee}(z)$ by
\begin{equation*}
F^{\vee}(z)\coloneqq\sum_{k=0}^n A^*_k z^{n-k}.
\end{equation*}

\begin{definition}
Let $\Omega$ be a $\mathbb C^{p\times p}$-valued positive semi-definite Borel matrix measure on $\mathbb R$. A sequence of matrix polynomials $(P_k(z))_{k=0}^m$ is called a sequence of {\it monic right orthonormal (or normalized orthogonal) matrix polynomials with respect to} $\Omega$ if deg $P_k(z)=k$ and 
$$
\int_{\mathbb R} P_k^{\vee}(u) {\rm d}\Omega (u) P_j(u)=\delta_{jk}I_p, \quad j,k=0,\ldots,m,
$$
where $\delta_{jk}$ stands for the Kronecker symbol.
\end{definition}

The application of orthogonal matrix polynomials can be found in  telecommuniation  \cite{GrIg}, information theory \cite{Fuhrmann}, matricial interpolation and moment problems \cite{HZC}, etc.

Let $\Omega$ be a $\mathbb C^{p\times p}$-valued positive semi-definite Borel matrix measure on $\mathbb R$.  
Suppose that $(P_k(z))_{k=0}^m$ is a sequence of monic $p\times p$ matrix polynomials with respect to $\Omega$. The orthogonality for $(P_k(z))_{k=0}^m$ with respect to $\Omega$ can be determined by  
the  moment sequence $({\bf s}_k)_{k=0}^{2m}$ of $\Omega$ as
\begin{equation}\label{Moments}
\int_{\mathbb R} u^k {\rm d} \Omega(u)={\bf s}_k,\quad k=0, \ldots, 2m.
\end{equation}

In fact, \eqref{Moments} guarantees that for $j,k=0,\ldots,m$ and $j\leq k$,
\begin{equation*}
\int_{\mathbb R} P_k^{\vee}(u) {\rm d}\Omega (u) P_j(u)
                                                       =[A^*_{kk}, \ldots,
                                                      A^*_{k0}] \begin{bmatrix}
            {\bf s}_{0} & \cdots & {\bf s}_{k}\\
            \vdots & \ddots  &  \vdots \\
            {\bf s}_{k} &  \cdots &  {\bf s}_{2k}
        \end{bmatrix} \begin{bmatrix} A_{jj} \\ \vdots\\
                                                       A_{j0}\\
                                                       \vdots\\
                                                       0_p \end{bmatrix},
\end{equation*}
where $ P_k(z):=\sum_{j=0}^k A_{k,k-j}z^j$, $A_{k0}=I_p$. The above relation indicates that
\begin{proposition}\label{ProOMPMoments}
Let $\Omega$ be a $\mathbb C^{p\times p}$-valued positive semi-definite Borel matrix measure on $\mathbb R$. Assume that $({\bf s}_j)_{j=0}^{2m}$ is the associated moment sequence 
of $\Omega$ as in \eqref{Moments}.
 $(P_k(z))_{k=0}^m$ is a sequence of monic right orthonormal matrix polynomials with respect to $\Omega$ if and only if the following equations hold for the coefficients of $P_k(z)$ for $k=0,\ldots,m$:
$$
 [A^*_{kk}, \ldots,
                                                      A^*_{k0}]\begin{bmatrix}
            {\bf s}_{0} & \cdots & {\bf s}_{k}\\
            \vdots &  \ddots  &  \vdots \\
            {\bf s}_{k} &  \cdots &  {\bf s}_{2k}
        \end{bmatrix} \begin{bmatrix} A_{kk} \\ 
                                                       \vdots\\
                                                       A_{k0} \end{bmatrix}=I_p,\quad
                                                       [A^*_{kk},\ldots,
                                                      A^*_{k0}] \begin{bmatrix}
            {\bf s}_{0} & \cdots & {\bf s}_{k-1}\\
            \vdots &  \ddots  &  \vdots \\
            {\bf s}_{k} &  \cdots &  {\bf s}_{2k-1}
        \end{bmatrix}=
0_{p\times kp},
$$ where $P_k(z):=\sum_{j=0}^k A_{k,k-j}z^j$.
\end{proposition}

 \cite[Theorem 2.3]{Du} points it out that these orthonormal matrix polynomials are special types of simple matrix polynomials:

\begin{proposition}\label{nullityzeroOMP} 
Let $\Omega$ be a $\mathbb C^{p\times p}$-valued positive semi-definite Borel matrix measure on $\mathbb R$. Suppose that $(P_k(z))_{k=0}^m$ is a sequence of monic right orthonormal matrix polynomials with respect to $\Omega$. Then each $P_k(z)$ is simple and  $\sigma(P_k)\subseteq \mathbb R$.  
\end{proposition}

\begin{proof}[Proof of Theorem \ref{ProHankelR}]
Suppose that $R(z)$ admits the Laurent expansion
\eqref{RationalLaurent} and $P(z):=\sum_{k=0}^{m} A_k z^{m-k}$ ($A_0=I_p$).  Let $(\wtilde {\bf s}_j)_{j=0}^{2m}$ be a $p\times p$ matrix sequence  given by
$$
\wtilde {\bf s}_j:=\begin{cases} -{\bf s}_j, & j=0, \ldots, 2m-1,\\
I_p+\sum_{j=0}^{m-1}A^*_{m-j} {\bf s}_{m+j}, & j=2m. 
\end{cases}
$$

Comparing  the matrix fraction  \eqref{RQP} and the Laurent expression  \eqref{RationalLaurent} of $R(z)$, one can find the equations
\begin{equation}\label{MFLE}
\begin{bmatrix}
        {\bf s}_j   & \cdots & {\bf s}_{j+m}  \\
        \vdots   & \ddots & \vdots\\
        {\bf s}_{j+k}  & \cdots& {\bf s}_{j+k+m}
    \end{bmatrix}
    \begin{bmatrix}
           A_{m} \\
         \vdots\\
         A_{0}\\
        \end{bmatrix}=0_{jp\times p},\quad j,k=0,1,\ldots,
\end{equation}
where $P(z):=\sum_{j=0}^m A_k z^{k-j}$.

From \eqref{MFLE} it is not difficult to calculate that
\begin{equation}\label{WHW}
W^*[\wtilde {\bf s}_{j+k}]_{j,k=0}^mW=\begin{bmatrix}
-{H}_{m-1}(R) & \\
& I_p
\end{bmatrix},
\end{equation}
where $W:=\begin{bmatrix}
I_p & & & A_m\\
   & \ddots &  & \vdots\\
   &  & I_p & A_1\\
   &  &    &  I_p
   \end{bmatrix}
$.
Since  ${\mathscr H}_{m-1}(R)$ is negative definite, the above equation means that $[\wtilde {\bf s}_{j+k}]_{j,k=0}^m$ is positive definite. In view of the solvability of truncated matricial Hamburger moment problems (see e.g. \cite{Kov,FKM,HZC}),  there exists at least a $p\times p$ positive semi-definite Borel matrix measure $\Omega$ on $\mathbb R$ such~that 
$$
\int_{\mathbb R} u^j {\rm d} \Omega(u)=\wtilde {\bf s}_{j},\quad j=0, \ldots, 2m.
$$
Suppose that $(P_k(z))_{k=0}^{m}$ is a sequence of orthonormal matrix polynomials  with respect to $\Omega$. 
A combination of \eqref{MFLE}, \eqref{WHW}  and Proposition \ref{ProOMPMoments} shows that $P(z)$ coincides with $P^{\vee}_{m}(z)$.
 Due to Proposition \ref{nullityzeroOMP}, $P(z)$ is simple and $\sigma(P)\subseteq \mathbb R$. The latter implies that $R(z)$ is holomorphic in~$\mathbb C_+$.

An application of  Lemma \ref{LemBezoutHankel}  yields that
\begin{equation*}
P(z)^{*} \frac{R(z)-R(z)^*}{z-\bar z} P(z)
=- \begin{bmatrix}I_p & \bar z I_p & \cdots & \bar z^{m-1}I_p \end{bmatrix}
    {\mathbf B}_{P^{\vee}, Q^{\vee}}(P, Q)
    \begin{bmatrix}I_p  \\ \vdots \\ z^{m-1}I_p \end{bmatrix}\succeq 0, \quad  z\in \mathbb C_+.
    \end{equation*} 
That is to say, $R(z)$ is a matrix-valued Herglotz-Nevanlinna function and, due to  Theorem \ref{ProRHankel}, $|\deg Q-\deg P|\leq 1$.

\end{proof}

As a consequence of Theorems \ref{ProRHankel} and \ref{ProHankelR}, we have
\begin{corollary}\label{CorMassE_j}
 Let $R(z)$ be a $p\times p$ self-adjoint rational matrix  function with the matrix fraction \eqref{RQP}. Suppose that $Q(z)$ and $P(z)$ are right coprime and $R(z)$ is  a matrix-valued Herglotz-Nevanlinna function with the form \eqref{RRepresent}. Then $P(z)$ is simple, $\sigma(P)\subseteq \mathbb R$ and
 \begin{equation}\label{EqEj}
E_j=-l_j Q(\lambda_j) \frac{({\rm adj}\ P)^{(l_j-1)}(\lambda_j)}{(\det P)^{(l_j)}(\lambda_j)}, \quad j=1,\ldots,r,
\end{equation}
where $\lambda_1, \ldots, \lambda_r$ are all the zeros of $P(z)$ with multiplicities $l_1,\ldots,l_r$, respectively.
Moreover, for $j=1,\ldots,r$, $E_j\neq 0_p$ and $\rank E_j=l_j$ whenever $P(z)$ and $Q(z)$ have no common zeros.
\end{corollary}

\begin{proof} The statement that $P(z)$ is simple and $\sigma(P)\subseteq \mathbb R$ is an immediate consequence of Theorems \ref{ProRHankel} and \ref{ProHankelR}.
In view of \cite[Subsections 4.3 and 4.5]{Lan}, the formula \eqref{EqEj} holds. 
 If in addition $P(z)$ and $Q(z)$ have no common zeros,
 \begin{equation*}
\rank E_j
=\rank ({\rm adj}\ P)^{(l_j-1)}(\lambda_j)=l_j, \quad j=1,\ldots,r,
\end{equation*}
where the last equation is due to  Proposition \ref{ProSimpleAdj} again.
\end{proof}

Coupled with Corollary \ref{CorMassE_j}, Proposition \ref{ProRRepresent} now fully extends the Chebotarev theorem to the matrix case. For the representation of $E_j$  in terms of the left and right latent vectors of $P(z)$ as in  Corollary \ref{CorMassE_j}, we refer the reader to \cite[Subsection 4.3]{Lan}.

It should be pointed out that for two $p\times p$ matrix polynomials $Q(z)$ and $P(z)$, the condition that they have no common zeros can imply their right coprimeness, while the converse implication does not generally hold (e.g. $
P(z)=\begin{bmatrix}
1 & 0\\
0 & z
\end{bmatrix}$ and  $Q(z)=\begin{bmatrix}
1 & 1\\
0 & z
\end{bmatrix}
$). 
Therefore, under the assumption of Theorem  \ref{ProHankelR}, $Q(z)$ and $P(z)$ are right coprime, while they may have common zeros:
\begin{example}
Given two $2\times 2$ matrix polynomials
$$
P(z):=\begin{bmatrix}
z & 1\\
1 & z
\end{bmatrix},\quad Q(z):=\begin{bmatrix}
4z-2 & -z+4\\
-z-1 & -z-2
\end{bmatrix},
$$
the $2\times 2$ rational matrix function $R(z):=Q(z)(P(z))^{-1}$ can be represented as
$$
R(z)
=\begin{bmatrix}
4 & -1\\
-1 & -1
\end{bmatrix}+ \frac{\begin{bmatrix} \frac12 & -\frac12\\
-\frac12 &\frac12\end{bmatrix}}{1-z}+\frac{\begin{bmatrix} \frac12 & -\frac12\\
-\frac12 &\frac12\end{bmatrix}}{-1-z}
=\begin{bmatrix}
4 & -1\\
-1 & -1
\end{bmatrix}+ \frac{\begin{bmatrix} -1 & 0\\
0 &-1 \end{bmatrix}}{z}+\cdots.
$$
It is readily checked that ${\mathscr H}_{-2,0}(R)=0_p$, ${\mathscr H}_{0}(R)=\begin{bmatrix} -1 & 0\\
0 & -1 \end{bmatrix}\prec 0$, $|\deg Q-\deg P|=0$ and $Q(z)$ and $P(z)$ are right coprime.
However, $1$ is a common zero of
 $Q(z)$ and $P(z)$.

\end{example}

\section{Main results}\label{SectionStability}
To extend the stability criterion via Herglotz-Nevanlinna functions,  we invoke  the stability criteria in terms of matricial Markov parameters.
An application of the extended Grommer theorem will be found to bridge these recently studied criteria and our extension.

\begin{definition}\label{DefEvenOddPart}
    A $p\times p$ matrix polynomial~$F(z)$ may be split into the
    \emph{even part}~$F_{e}(z)$ and the \emph{odd part}~$F_{o}(z)$ so that
    \( % I want this formula to appear in relation to this definition: many will find it more
       % illustrative than the explicit expressions below.
        F(z)=F_{e}(z^2)+zF_{o}(z^2).
    \)
    For~$F(z)$  written as in~\eqref{ReprentationMatrixPolynomials},
    they are defined by
    \begin{equation*}
  F_{e}(z)\coloneqq \sum_{k=0}^m A_{2k}z^{m-k}
  \quad\text{and}\quad
  F_{o}(z)\coloneqq \sum_{k=1}^m A_{2k-1}z^{m-k}
\end{equation*}
when~$\deg F=2m$, and by
\begin{equation*}
    F_{e}(z)\coloneqq \sum_{k=0}^m A_{2k+1}z^{m-k}
    \quad\text{and}\quad
    F_{o}(z)\coloneqq \sum_{k=0}^m A_{2k}z^{m-k}
\end{equation*}
when~$\deg F=2m+1$.
\end{definition}

\begin{remark}
For a monic $p\times p$ matrix polynomial of even (resp. odd) degree, its even (resp. odd)  part   is monic as well.
\end{remark}

Given a monic $p\times p$ matrix polynomial $F(z)$ with  the even part $F_e(z)$  and the
odd part $F_o(z)$, we associate with a pair of $p\times p$  rational matrix functions   
\begin{equation}\label{R_F}
R_F(z)\coloneqq 
         -\dfrac{F_{o}(z)}{F_{e}(z)}, \quad R_{z,F}(z)\coloneqq -zR_F(z)=  \dfrac{zF_{o}(z)}{F_{e}(z)}, \quad z\in \mathbb C\setminus \sigma(F_e)
\end{equation}         
when $F_e(z)$ is regular and another pair of $p\times p$  rational matrix functions        
\begin{align}             
 &\wtilde R_F(z)\coloneqq  \dfrac{F_{e}(z)}{F_{o}(z)}, \quad z\in \mathbb C\setminus \sigma(F_o), \label{tildeR_F} \\  
&\wtilde R_{z,F}(z)\coloneqq -z^{-1}\wtilde R_F(z)= -\dfrac{F_{e}(z)}{zF_{o}(z)},  \quad z\not\in \sigma(F_o)\cup\{0\}. \label{tildeRzF}
\end{align}
when $F_o(z)$ is regular. Obviously when both $F_e(z)$ and $F_o(z)$ are regular, 
$$
R_F(z)=-(\wtilde R_F(z))^{-1},\quad z\not\in \sigma(F_e)\cup \sigma(F_o),
$$  
and 
$$
R_{z,F}(z)=-(\wtilde R_{z,F}(z))^{-1},\quad z\not\in \sigma(F_e)\cup \sigma(F_o)\cup \{0\}.
$$
The aforementioned rational matrix functions play a role in testing the stability of $F(z)$:
 
\begin{theorem}\label{ThmStabilityMarkov} \cite[Theorem 4.4]{XD} Let $F(z)$ be a monic $p\times p$ matrix polynomial with  the even part $F_e(z)$  and the
odd part $F_o(z)$.
\begin{enumerate}[(a)]
\item For $\deg F=2m$, suppose that  $R_F(z)$ and $R_{z,F}(z)$ are given as \eqref{R_F}. Assume that $R_F(z)$ is self-adjoint. The Hurwitz stability for $F(z)$ is equivalent to the negative definiteness of $\mathscr H_{m-1}(R_F)$ and $\mathscr H_{m-1}(R_{z,F})$.
\item For $\deg F=2m+1$, suppose that  $\wtilde R_F(z)$ and $\wtilde R_{z,F}(z)$ are given as \eqref{tildeR_F}--\eqref{tildeRzF}. Assume that $\wtilde R_F(z)$ is self-adjoint. The Hurwitz stability for $F(z)$ is equivalent to  the negative definiteness of $\mathscr H_{m-1}(\wtilde R_{F})$ and $\mathscr H_{m}(\wtilde R_{z,F})$.
\end{enumerate}
\end{theorem}

  This theorem, called \textit{stability criteria via matricial Markov parameters}, generalizes the corresponding classical results for real scalar polynomials (see Theorem~17 of~\cite[Chapter XV]{Gan}). 
For a monic matrix polynomial $F(z)$ of odd degree satisfying that $F_e(z)$ is regular,  two $p\times p$ rational matrix functions $R_F(z)$ and $R_{z,F}(z)$ 
 as in \eqref{R_F} are well-defined and self-adjoint as well. In this case,  an alternative version of the above stability criteria in the odd case  appear:

\begin{theorem}\label{ThmAlterOdd} \cite[Theorem 4.10]{XD} Under the assumption of the above theorem for  $\deg F=2m+1$, let $F_e(z)$ be regular.  The Hurwitz stability for $F(z)$ is equivalent to the negative definiteness of $\mathscr H_{m-1}(R_F)$, $\mathscr H_{m-1}(R_{z,F})$ and of the limit $\lim\limits_{z\rightarrow \infty} R_F(z)$.
\end{theorem}

More detailed consideration on these criteria can be found in \cite[Section 4]{XD}.  We add here a new criterion to address the case when the degree of  $F(z)$ is even:

\begin{theorem}\label{ThmAlterEven} Under the conditions of Theorem \ref{ThmStabilityMarkov} for $\deg F=2m$,  let $F_o(z)$ be regular. The Hurwitz stability for $F(z)$ is equivalent to the negative definiteness of $\mathscr H_{m-2}(\wtilde R_F)$, $\mathscr H_{m-1}(\wtilde R_{z,F})$ and of the limit $\lim\limits_{z\rightarrow \infty} \wtilde R_{z,F}(z)$.
\end{theorem}

\begin{proof} Suppose that $\wtilde A:=\lim\limits_{z\rightarrow \infty}  \wtilde R_{z,F}(z) $, $F_o(z)=:\sum_{k=0}^{m-1} O_{m-1-k} z^k$ and $F_e(z)=:\sum_{k=0}^{m} E_{m-k} z^k$. 
Using Lemma \ref{LemBezoutHankel} we have
\begin{equation*}
 B_{F_o^{\vee},F_{e}^{\vee}}(F_o,F_{e})=\begin{bmatrix}
O_{m-1}^* & \cdots & O_{0}^*\\
\vdots & \iddots & \\
O_{0}^* &&
\end{bmatrix}\cdot
\begin{bmatrix}\wtilde A & \\ & \mathscr H_{m-2}(\wtilde R_F) \end{bmatrix}\cdot
\begin{bmatrix}
O_{m-1} & \cdots & O_{0}\\
\vdots & \iddots & \\
O_{0} &&
\end{bmatrix}
\end{equation*}
and, on the other hand,
\begin{equation*}
B_{F_o^{\vee},F_{e}^{\vee}}(F_o,F_{e})
=B_{F_e^{\vee},-F_{o}^{\vee}}(F_e,-F_{o})
=\begin{bmatrix}
E_{m-1}^* & \cdots & E_{0}^*\\
\vdots & \iddots & \\
E_{0}^* &&
\end{bmatrix}\cdot
 \mathscr H_{m-1}(R_{F}) \cdot
\begin{bmatrix}
E_{m-1} & \cdots & E_{0}\\
\vdots & \iddots & \\
E_{0} &&
\end{bmatrix}.
\end{equation*}
Therefore, the negative definiteness of both $\wtilde A$ and  $\mathscr H_{m-2}(\wtilde R_F)$ is equivalent to that of  $\mathscr H_{m-1}(R_F)$. Analogously, the negative definiteness of $\mathscr H_{m-1}(\wtilde R_{z,F})$ is equivalent to that of  $\mathscr H_{m-1}(R_{z,F})$. Then the proof is complete due to (a) of Theorem \ref{ThmStabilityMarkov}.
\end{proof}

Now we reshape these stability criteria for matrix polynomials $F(z)$ in terms of the Herglotz-Nevanlinna property for $R_F(z)$,  which differs with regard to the degree of $F(z)$ by the same token:

\begin{theorem}\label{ThmHurwitz}
Let $F(z)$ be  a monic $p\times p$ matrix polynomial with  the even part $F_e(z)$  and the
odd part $F_o(z)$. Assume that $F_e(z)$ is regular when $\deg F$ is odd and $R_F(z)$ is the self-adjoint rational matrix function as  \eqref{R_F}.
$F(z)$ is Hurwitz stable if and only if the following conditions are simutaneously true:
 \begin{enumerate}
  \item[(a)]  $R_F(z)$ is  a matrix-valued Herglotz-Nevanlinna function;
 \item[(b)] $F_e(z)$ and $F_o(z)$ are right coprime;
   \item[(c)] 
All zeros of $F_e(z)$ are negative real;
  \item[(d)] $F_o(z)$ is regular when $\deg F$ is even;
   \item[(e)] $\lim\limits_{z\rightarrow \infty} R_F(z)\prec 0$
   when $\deg F$ is odd.
\end{enumerate}
In this case, $F_e(z)$ is simple.
\end{theorem}

\begin{proof}  
We firstly give a proof for the case that $\deg F$ is even.  

Proposition \ref{ProRRepresent} shows that $R_F(z)$ has the form \eqref{RRepresent}.
 By  Lemma 6.3-11 of \cite{Kai}, $R_F(z)$ is strictly proper.  So both coefficients $A$ and $B$ in  the expression \eqref{RRepresent}  of $R_F(z)$ equal  $0_p$. 
By setting the matrix function 
$R_{z,F}(z)$ as in \eqref{R_F},
we have
\begin{equation}\label{RzFRepresent}
R_{z,F}(z)=-\sum_{j=1}^r E_j+ \sum_{j=1}^r  \dfrac{-\lambda_j E_j}{\lambda_j-z}.
\end{equation}

The proof for the ``if'' implication: Suppose that the statements {\rm (a)}--{\rm (d)} hold. The statement (c) and Proposition \ref{ProRRepresent} imply  that  $\{\lambda_j\}_{j=1}^r$ coincides with $\sigma(F_e)$ and $-\lambda_j> 0$. Using Proposition \ref{ProRRepresent} again, we see that 
 $R_{z,F}(z)$ is  also a matrix-valued Herglotz-Nevanlinna function. 
 Moreover, the right coprimeness of $zF_o(z)$ and $F_e(z)$ follows from (b), (c) and Lemma 6.3-6 in \cite{Kai}.  Thus an application of Theorem \ref{ProRHankel} and (a) of Theorem  \ref{ThmStabilityMarkov}  completes the proof. 
  
The proof for the ``only if'' implication:  Suppose that  $F(z)$ is Hurwitz stable. A combination of Theorem \ref{ProHankelR} and  Theorem  \ref{ThmStabilityMarkov} (a) reveals 
the statements  (a)--(b), the Herglotz-Nevanlinna property of $R_{z,F}(z)$ and  that $F_e(z)$ is simple. 
The statement (c) is proved by contradiction:
Assume that there exists a zero $\lambda_j$ of $F_e(z)$ with multiplicity $l_j$ such that $\lambda_j\geq 0$. As is seen from \eqref{RzFRepresent} and  Proposition \ref{ProRRepresent}, $R_{z,F}(z)$ cannot be a matrix-valued Herglotz-Nevanlinna function.

With the help of Theorem \ref{ThmAlterOdd}, the proof for the odd case  can be conducted analogously to that  for the even case.

\end{proof}

When $F(z)$ is a real scalar polynomial, Theorem \ref{ThmHurwitz} coincides with the classical stability criterion via Herglotz-Nevanlinna functions. On the other hand, to apply this matricial version in the odd case, the tested matrix polynomial needs to satisfy a precondition that its even part is regular. To avoid this restriction, we give the following alternative extension which also covers the even case:

\begin{theorem}\label{ThmHurwitzodd}
Let $F(z)$ be  a monic $p\times p$ matrix polynomial with  the even part $F_e(z)$  and the
odd part $F_o(z)$.   
Assume that $F_o(z)$ is regular when $\deg F$ is even and $\wtilde R_{z,F}(z)$ is the self-adjoint rational matrix function as  in \eqref{tildeRzF}.
 $F(z)$ is Hurwitz stable if and only if
the following statements are simutaneously true:
 \begin{enumerate}
    \item[(a)] $\wtilde R_{z,F}(z)$ is  a matrix-valued Herglotz-Nevanlinna function;
       \item[(b)] $F_e(z)$ and $F_o(z)$  are right coprime;
    \item[(c)]   All zeros of  $F_o(z)$ are negative real;
    \item[(d)]  $0\not\in \sigma(F_e)$;
     \item[(e)] 
$\lim\limits_{z\rightarrow \infty} \wtilde R_{z,F}(z)\prec 0$
   when $\deg F$ is even.
\end{enumerate}
In this case, $F_o(z)$ is simple.
\end{theorem}

\begin{proof}
The  proof can be conducted analogously to that of Theorem \ref{ThmHurwitz}, in placing the matrix function $R_F(z)$ as in \eqref{R_F} with $\wtilde R_{z,F}(z)$ as in \eqref{tildeRzF}: For the Hurwitz stability of $F(z)$, by Theorem \ref{ThmStabilityMarkov} (b) (for the odd case) and Theorem \ref{ThmAlterEven} (for the even case) it is equivalent that the statements (a) and (e) hold and
 \begin{enumerate}
       \item[{\rm (b')}] $F_e(z)$ and $zF_o(z)$  are right coprime;
    \item[{\rm (c')}]   All zeros of  $zF_o(z)$ except for $0$ are negative real;
    \item[{\rm (d')}]  $F_e(z)$ is regular when $\deg F$ is odd.
\end{enumerate}
In this case, both $F_o(z)$ and $zF_o(z)$ are simple. Then the assertion is evident, since the statements (b') and (d') are equivalent to (b) and (d) due to  Lemma 6.3-6 in \cite{Kai} and the statement (c) implies (c') and, conversely, follows from (c') and the statement that $F_o(z)$ is simple. 
\end{proof}

From the proofs of Theorems \ref{ThmHurwitz} and \ref{ThmHurwitzodd} one can see that
\begin{corollary}\label{LemHurwitzNevanlinna} 
    Let $F(z)$ be a monic $p\times p$ matrix polynomial  with  the even part $F_e(z)$  and the
odd part $F_o(z)$. 
\begin{enumerate}
\item[(a)] For $\deg F=2m$, suppose that $R_F(z)$ and $R_{z,F}(z)$ are the self-adjoint rational matrix functions as  \eqref{R_F}. $F(z)$ is Hurwitz stable if and only if $R_F(z)$ and $R_{z,F}(z)$ are both matrix-valued Herglotz-Nevanlinna functions, $F_e(z)$ and $zF_o(z)$ are right coprime and $F_o(z)$ is regular. 
\item[(b)] For $\deg F=2m+1$, suppose that  $\wtilde R_F(z)$ and $\wtilde R_{z,F}(z)$ are the self-adjoint rational matrix functions as  \eqref{tildeR_F}--\eqref{tildeRzF}. $F(z)$ is Hurwitz stable if and only if $\wtilde R_F(z)$ and $\wtilde R_{z,F}(z)$ are both matrix-valued Herglotz-Nevanlinna functions, $F_e(z)$ and $zF_o(z)$ are right coprime and $F_e(z)$ is regular.
\end{enumerate}
    \end{corollary}

In view of \cite[P. 11]{XD}, given a stable monic matrix
polynomial~$F(z)$ of odd degree, $\wtilde R_{z,F}(z)$  of the form  \eqref{tildeRzF} is  self-adjoint if and only if the even part $F_e(z)$ is regular and $R_F(z)$  of the form  \eqref{R_F} is  self-adjoint as well. In other words, a matrix polynomial of odd degree whose stability can be tested via
Theorem~\ref{ThmHurwitzodd}, but not via Theorem~\ref{ThmHurwitz},
cannot be Hurwitz stable. Analogously,  a matrix polynomial of even degree whose stability can be tested via
Theorem~\ref{ThmHurwitz}, but not via Theorem~\ref{ThmHurwitzodd},
cannot be Hurwitz stable. Moreover, some common features of $F_e(z)$ and the odd part $F_o(z)$ appear:

\begin{corollary}\label{CorHurwitzZero}
Under the assumption of Theorem~\ref{ThmHurwitz} or \ref{ThmHurwitzodd},
let $F(z)$ be Hurwitz stable. Then $F_e(z)$ and $F_o(z)$ are right coprime and simple, of which  all zeros are negative real.
\end{corollary}

This fact may motivate one to consider the possible extension of the following modification to Hermite-Biehler theorem \cite[Theorem 14, Chapter XV, p228]{Gan}.
\begin{theorem*}[\textit{Modified Hermite-Biehler theorem}]
Let $f(z)$ be a real scalar  polynomial of degree $n=2m$ (resp. $n=2m+1$) with the even part $f_{e}(z)$ and odd part $f_{o}(z)$.  $f(z)$  is  Hurwitz stable if and only if the highest coefficients of $f_{e}(z)$ and $f_{o}(z)$ have the same sign, and the roots of $f_{e}(z)$ and $f_{o}(z)$ are real negative, simple and interlacing, that is, by denoting the spectrums $\sigma (f_e):=\{\lambda_{e,k}\}_{k=1}^m$ and $\sigma (f_o):=\{\lambda_{o,k}\}_{k=1}^{m-1}$ (resp. $\sigma (f_o):=\{\lambda_{o,k}\}_{k=1}^m$),
$$
\lambda_{e,1}<\lambda_{o,1}<\lambda_{e,2}<\cdots<\lambda_{o,m-1}<\lambda_{e,m} \quad (\mbox{resp. } \ \lambda_{e,1}<\lambda_{o,1}<\lambda_{e,2}<\cdots<\lambda_{e,m}<\lambda_{o,m}).
$$
\end{theorem*}
Unfortunately in the matrix case $F_e(z)$ and $F_o(z)$ may share common zeros even when $F(z)$ is Hurwitz stable, as is illustrated in the following examples.

\begin{example}
Let a monic $2\times 2$ matrix polynomial of degree $3$
$$
F(z):=\begin{bmatrix}
z^3+ 18 z^2+108 z+216 & 0\\
0 & z^3+3 z^2+12 z+20
\end{bmatrix},
$$
where the odd part $
F_o(z)$ and the even part $F_e(z)$ are written as
$$
F_o(z)=\begin{bmatrix}
z+108 & 0\\
0 & z+12
\end{bmatrix},\quad 
F_e(z)=\begin{bmatrix}
18 z+216 & 0\\
0 & 3 z+20
\end{bmatrix}.
$$ 
By calculation, $\sigma(F)=\{-6,-2,-\frac{1}{2}\pm \frac{\sqrt{39}}{2}{\rm i} \}$ and $F(z)$ is Hurwitz stable.
On the other hand, obviously the conclusion of Corollary \ref{CorHurwitzZero} is fulfilled for this example and the rational matrix function
$$
R_F(z):=-\dfrac{F_{e}(z)}{z F_{o}(z)}=\frac{\begin{bmatrix} 2 & 0\\ 0 & \frac{5}{3}\end{bmatrix}}{-z}+\frac{\begin{bmatrix} 16 & 0\\ 0 & 0\end{bmatrix}}{-108-z}+\frac{\begin{bmatrix} 0 & 0\\ 0 & \frac{4}{3}\end{bmatrix}}{-12-z}
$$
is a matrix-valued Herglotz-Nevanlinna function.
However, $F_e(z)$ and $F_o(z)$ share a common zero $-12$. 

\end{example}

\section*{Acknowledgement}
The author is grateful to Alexander Dyachenko  who inspired the author to this intersting topic, suggested a few corrections in Proposition \ref{ProRRepresent} and provided a number of valuable modifications for the whole manuscript. This work was supported by the Foundation for Fostering Research of Young Teachers in South China Normal  University (Grant No. 19KJ20).

\appendix

\bibliographystyle{amsplain}

\end{document}